\documentclass[11pt]{article}

\usepackage{amsmath, amssymb, amsthm}
\usepackage{hyperref}
\usepackage{enumitem}
\usepackage[numbers]{natbib}
\usepackage[margin=1.2in]{geometry}

\title{Note on the Equivalence of Costas Polynomials and Orthomorphisms}

\author{
Amela~Muratovi\'{c}-Ribi\'{c}\thanks{Department of Mathematics and Computer Science, Faculty of Science, University of Sarajevo, Bosnia and Herzegovina. \\ Emails: amela@pmf.unsa.ba; aleksandar.bs@pmf.unsa.ba.}
\and
Aleksandar~Bala\v{s}ev-Samarski\footnotemark[1]}

\date{June 5, 2026}

\newtheorem{Them}{Theorem}[section]
\newtheorem{Prop}{Proposition}[section]

\newtheorem{Cor}{Corollary}[Them]
\newtheorem{Def}{Definition}[section]
\newtheorem*{Conj*}{Conjecture}

\newcommand{\keywords}[1]{\par\noindent
  \textbf{Keywords: } #1}

\newcommand{\msc}[1]{\par\noindent
  \textbf{MSC 2020: } #1}

\begin{document}

\maketitle

\begin{abstract}
    We establish an equivalence between the existence of Costas polynomials and the existence of a special kind of orthomorphism such that their compositions are also orthomorphisms. Computations are easier over these orthomorphisms. We provide a lower bound for the number of Costas polynomials and derive some of their properties. We show that Costas polynomials, by virtue of being multiplicative analogs of planar polynomials, can also be used to construct complete families of mutually orthogonal Latin squares.
\end{abstract}

\keywords{
    Costas polynomial; permutation polynomial; Latin square; orthomorphism
}

\msc{11T06; 12E20; 05B15}

\section{Introduction}

\par Let $p$ be a prime, and $q=p^n$ where $n$ is a positive integer. Denote by $\mathbb{F}_q$ the finite field of order $q$. Every mapping $\theta$ on $\mathbb{F}_q$ induces a polynomial $f(x)\in \mathbb{F}_q[x]$ of the degree  less than $q$ such that $\theta(s)=f(s)$ for all $s\in \mathbb{F}_q[x]$.  If this mapping is a permutation of $\mathbb{F}_q$, then we say that $f(x)$ is \emph{ a permutation polynomial} (PP). Polynomials of the form $L(x)=\sum_{i=0}^{n-1}a_ix^{p^i}$ are called \emph{linearized permutation polynomials} since these are linear operators of $\mathbb{F}_q$ over $\mathbb{F}_p$ (see \citep{Lidl}).
\par A Costas array of order $n$ is a $n \times n$ permutation array (with exactly one dot in every row and column and blanks elsewhere) such that every vector connecting two dots is distinct. The Costas property ensures that the array has ideal auto-correlation, which makes Costas arrays highly desirable for use in RADAR and SONAR communications. 
 A special case is the circular Costas sequence, which motivated the definition of Costas polynomials. For more information, see \citep{[2],[3],[9], David,[1]}. 

\begin{Def}
Suppose $f \in \mathbb{F}_q[x] $ is such that $f(0)=0$ and $f(dx)-f(x)$ is a permutation polynomial for all $d \in \mathbb{F}_q$, $d \neq 1$. Then $f$ is called a Costas polynomial.
\end{Def}
Without lost of generality we can assume $f(0)=0$.
According to an existing conjecture (see \citep{[9]}) the only circular Costas sequence is exponential Welch, or, equivalently, the only Costas polynomial over the prime field  is of the form $x^s$, with $\gcd(s,q-1)=1$. 
Known classes of Costas polynomials include $f(x)=x^s$ where $\gcd(q-1,s)=1$ for any field. In the case where $\mathbb{F}_{q^m}$ is an extension field and $L(x)=\sum_{i=0}^n a_i x^{q^i}$ is a linearized permutation polynomial of $\mathbb{F}_{q^m}$ over $\mathbb{F}_q$, then $L(x^s)$ is also a Costas polynomial for all $s$ such that $\gcd(s,q^m-1)=1$. It is conjectured in \citep{David} that these are the only Costas polynomials.

\par A \emph{Latin square} of order $n$ on a set $S$, where $|S|=n$,  is a table (matrix) of dimension $n \times n$ with elements from $S$ such that in each row and each column the elements are distinct. Usually $S$ is a finite group with $n$ elements and $n$ rows, and columns are labeled by elements of the group. Latin squares are widely used to make schedules, in design theory, etc. (see \citep{Evans}) . The element in the $i$-th row and the $j$-th column of the Latin square $\mathcal{L}$ is denoted by $\mathcal{L}_{i,j}$. Two Latin squares are said to be equivalent if the second Latin square can be obtained from the first by permuting the rows, columns, and elements. Two Latin squares $\mathcal{L}$ and $\mathcal{K}$ of the same order $n$ are said to be \emph{mutually orthogonal} if the ordered pairs $(\mathcal{L}_{i,j}, \mathcal{K}_{i,j})$ are distinct for all $1\leq i,j\leq n $. The maximum number of pairwise mutually orthogonal Latin squares of order $n$ is $n-1$, and they can be constructed over finite fields of order $n=q$ by $\mathcal{L}^a_{i,j}=i+aj$ where $i,j \in \mathbb{F}_q$, $a\in \mathbb{F}_q^*$. Let $f \colon G\rightarrow G$. By Mann (\citep{Mann}), the Latin squares $\mathcal{L}_{i,j}=i+j$ and $\mathcal{K}_{i.j}=i+f(j)$, $i,j\in G$ are mutually orthogonal if both $f(x)$ and $f(x)-x$ are permutations of the group $G$. 
\begin{Def} Let $G$ be a finite group, and $\theta \colon G\rightarrow G$. We say that $\theta $ is an \emph{orthomorpism} if both $\theta(x)$ and $\theta(x)-x$ are permutations on $G$.
\end{Def}
We say that $f(x)$ is an orthomorphism in a finite field $\mathbb{F}_q$ if it is an orthomorphism in its additive group $(\mathbb{F}_q,+)$.

\section{Properties of Orthomorphisms Derived from Costas Polynomials}

 Let $f(x)$ be a Costas polynomial over $\mathbb{F}_q$, i.e.\ $f(0)=0$ and $f(dx)-f(x)$ are permutation polynomials (PP) for all $d \in \mathbb{F}_q \setminus \{1\}$. In particular, if $d=0$ then it follows that $f(x)$ is a PP. Using the substitution $x=f^{-1}(y)$ in the expression $f(dx)-f(x)$, we obtain that $f(df^{-1}(y))-y$ is a PP for all $d \in \mathbb{F}_q \setminus \{1\}$. 

Let $d \in \mathbb{F}_q^*$ and denote $g_d(x)=f(df^{-1}(x))$. As a composition of PPs, $g_d(x)$ is again a PP. Therefore, both $g_d(x)$ and $g_d(x)-x$ are permutations of $\mathbb{F}_q$ and $g_d(x)$ is an orthomorphism for $d\in \mathbb{F}_q\setminus \{1\}$.  Denote by $g^n(x)=g \circ g \circ \cdots \circ g(x)$ the $n$-fold composition of $g(x)$ with itself. Then
$
g_d^2(x)=f\big(d f^{-1}(f(d f^{-1}(x)))\big)=f(d^2 f^{-1}(x))=g_{d^2}(x),
$
and by induction it follows that $g_d^n(x)=g_{d^n}(x)$ is also an orthomorphism.

\begin{Them}
There exists a Costas polynomial $f(x)$ over $\mathbb{F}_q$ if and only if there exists an orthomorphism $g(x)$ such that $g(0)=0$ and all of its compositions $g^k(x)$ are also orthomorphisms for $k=1,2,\ldots,q-2$. They are related through the relation $g(x)=f(\alpha f^{-1}(x))$, where $\alpha $ is a primitive element in $\mathbb{F}_q$.

\par Moreover, all Costas polynomials are of the form $L(x^s)$ where $L(x)$ is a linearized permutation and $\gcd(s,q-1)=1$ if and only if the only orthomorphisms such that $g(0)=0$ and all its compositions $g^k(x)$, $k=1,2,\ldots,q-2$, are orthomorphisms of the form $g(x)=L(\alpha L^{-1}(x))$ where $L(x)$ is a linearized permutation and $\alpha $ is a primitive element in $\mathbb{F}_q$. 
\end{Them}
\begin{proof}
Assume $f(x)$ is a Costas polynomial and define $g(x)=g_\alpha(x)=f(\alpha f^{-1}(x))$ where $\alpha$ is a primitive element in $\mathbb{F}_q$. From $f(0)=0$, it follows $g(0)=0$. Then we have $g^k(x)=g_{\alpha^k}(x)=f(\alpha^k f^{-1}(x))$ and $g^k(x)=g_{\alpha^k}(x)-x$  are all PPs for $1 \leq k \leq q-2$. Therefore, the compositions $g^k(x)$ are orthomorphisms for all $1 \leq k \leq q-2$.

\par Conversely, assume now that there exists an orthomorphism $g(x)$ such that $g(0)=0$ and all $g^k(x)$, $1 \leq k \leq q-2$, are orthomorphisms. Every orthomorphism $\varphi$ has exactly one fixed point (since $\varphi (x)-x$ is a PP), thus from $g(0)=0$ follows that $0$ is the only fixed point of the permutation $g(x)$. If $x_0 \neq 0$ lies in a cycle of length $t<q-1$ in the permutation $g(x)$ of $\mathbb{F}_q$, then $x_0=g^t(x_0)$, which implies $g^t(x_0)-x_0=0=g^t(0)-0$, contradicting the fact that $g^t(x)-x$ is a PP. Therefore, the permutation $g(x)$ has one fixed element $0$ and one cycle of length $q-1$. Let $\alpha$ be a primitive element of $\mathbb{F}_q$. Then the PP $h(x)=\alpha x$ has the same cycle structure. Thus, $g(x)$ and $h(x)$ are conjugate; i.e.\ there exists a permutation $f(x)$ such that $g=f \circ h \circ f^{-1}$,  or equivalently $g(x)=f(\alpha f^{-1}(x))$. Using $g^k(x)=f(\alpha^k f^{-1}(x))$, $1 \leq k \leq q-2$, we see that $f(\alpha^k f^{-1}(x))-x$ are permutations. Substituting $y=f^{-1}(x)$, we obtain that $f(\alpha^k y)-f(y)$ are all permutations for $1\leq k\leq q-2$. Therefore, $f(dx)-f(x)$ are permutations for $d \neq 1$; i.e., $f(x)$ is a Costas polynomial.
\par 
Assume now the Costas polynomial $f(x)$ is of the form $f(x)=L(x^s)$. Then $g_\alpha (x)=f(\alpha f^{-1}(x))=L(\alpha ^s (L^{-1}(x)))$, which is a linearized polynomial.
Conversely, assume that $g_\alpha (x)=f(\alpha f^{-1}(x))$ is of the form $L(\beta L^{-1}(x))$ where $\beta =\alpha^s, \gcd(s,q-1)=1$, $TODO: \beta$ is a primitive element of $\mathbb{F}_q$, i.e.\ $L(\beta L^{-1}(x))=f(\alpha f^{-1}(x))$. Then $\alpha^s(L^{-1}\circ f)(y)=(L^{-1}\circ f)(\alpha y)$. Comparing the coefficients of the polynomials on both sides we conclude that $L^{-1}\circ f$ is a monomial, more precisely $L^{-1}\circ f (y)=y^s$. Thus $f(y)=L(y^s)$.
\end{proof}
\begin{Cor} The orthomorphism $g(x)$ has one fixed point $0$ and one cycle of length $q-1$.
\end{Cor}
We now show that the orthomorphism $g(x)$ corresponds to a Costas polynomial if and only if $g^k(x)$ corresponds to a Costas polynomial for all $k$ such that $\gcd(k,q-1)=1$.
\begin{Cor} Let $g(x)$ be such that $g(0)=0$ and has one cycle of length $q-1$. It then follows that all the compositions $g^s(x), 1\leq s\leq q-2$, are orthomorphisms if and only if all the $s$-folded  compositions $g^{ks}(x)$ of $g^k(x)$ with $1\leq s\leq q-2$ are orthomorphisms  for all $k$ such that $\gcd(q-1,k)=1$. 
In particular, $g^{-1}(x)$ is an orthomorphism, $g(x)$ is such that $g(0)=0$, it has one cycle of the length $q-1$ and all of its compositions $g^s(x), 1\leq s\leq q-2$ are orthomorphisms.
\end{Cor}
\begin{proof} From $\gcd(q-1,k)=1$ if follows that the cycle structures of $g(x)$ and $g^k(x)$ are the same. Moreover, since the $g^s(x)$ are all orthomorphisms (or all are non-orthomorphisms) for $s=1,2,\ldots ,q-2$ and $g^{q-1}(x)=\mathrm{id}(x)=x$, it follows that $(g^k)^s(x)=g^{ks}(x)=g^{ks\pmod{ q-1}}(x)$ is an orthoomorphism (or is a non-orthomorphisms) for all $s=1,2,\ldots ,q-2$, and conversely. Note that $g^{-1}(x)=g^{q-2}(x)$ As $\gcd(q-1,q-2)=1$ it follows that $g^{-1}(x)$ is an orthomorphism of the given property.\end{proof}

Using the fact that $L(x^s)$ is Costas, we obtain:

\begin{Cor}
Let $L(x)$ be a linearized polynomial over $\mathbb{F}_q$, and let $d$ be a primitive element. Then $g(x)=L(dL^{-1}(x))$ satisfies $g(0)=0$. It has one cycle of length $q-1$, and all of its compositions $g^k(x)$, for $1\leq k\leq q-2$, are orthomorphims. Moreover, $g^{q-1}(x)=x$.
\end{Cor}

\begin{Them}\label{thm:22} The lower bound for the number of Costas polynomials over the field $F_q$, where $q=p^n$, is 
$$\varphi(q-1)\frac{(q-1)(q-p)\ldots (q-p^{n-1})}{n(q-1)}.$$
\end{Them}
\begin{proof} All orthomorphims of the form $L(\alpha L^{-1}(x))$, where $\alpha $ is a primitive element in $\mathbb{F}_q$ and $L(x)$ is a linearized permutation polynomial, produce a Costas polynomial. There are $\varphi(q-1)$ primitive elements in $\mathbb{F}_q$, and there are $(q-1)(q-p)\ldots (q-p^{n-1})$ linearized permutation polynomials over $\mathbb{F}_q$. However,
$L(\alpha L^{-1}(x))=L_1(\beta L_1^{-1}(x))$ is equivalent to $(L_1^{-1}\circ L)(\alpha x)=\beta (L_1^{-1}\circ L)(x)$. Therefore, $L_1^{-1}\circ L(x)=ax^{p^j}$ for some $a\in \mathbb{F}_q^*$ and $j\in \{0,1,\ldots ,n-1\}$. This implies that $(q-1)n$ choices of $L(x)$ and $\alpha $ yield the same orthomorphism. Consequently, the number of distinct orthomorphisms (and thus Costas polynomials of this form) is
$\varphi(q-1)\frac{(q-1)(q-p)\ldots (q-p^{n-1})}{n(q-1)}$,
where $\varphi(n)$ denotes the Euler totient function. 
\end{proof}
Note that the same result can also be obtained by counting the $L(x^s)$.
We conclude from the above that the conjecture proposed in \citep{David}, namely that the only Costas polynomials are those of the form $L(x^s)$, $\gcd(q-1,s)=1$ is equivalent to the following conjecture:
\par
\begin{Conj*}
The only mappings $g(x)$ with fixed point $0$ and a cycle of length $q-1$, and such that all of their compositions $g^k(x)$, $1\leq k\leq q-2$ are orthomorphisms, are those of the form $L(\alpha L^{-1}(x))$ where $L(x)$ is a linearized permutation polynomial and $\alpha $ is a primitive element of $\mathbb{F}_q$. 
\end{Conj*}
\par  
Computationally, it is easier to find all orthomorphisms $g$, with $g(0)=0$ and a single cycle of length $q-1$, such that all their compositions are orthomorphisms, due to the specific cycle structure of $g(x)$, than to investigate which permutation polynomials are Costas. Note that there are $(q-1)!$ permutation polynomials with fixed point $0$, and one must store the full table of preimages and images to test whether a permutation is Costas. On the other hand, there are $(q-2)!$ permutations $g$ with $g(0)=0$ and having one cycle of length $q-1$, and it suffices to store only the cycle. Using the fact that $g^k(x), \gcd(k,q-1)=1$ preserves the property of being (or not being) an orthomorphism of the specified type, we need to investigate $(q-2)!/\varphi(q-1)\qquad $ $q-1$-cycles to determine whether all of their compositions are orthomorphisms. 
Assume that for a subarray in the $q-1$-cycle $(\ldots ,c_i,c_{i+1}\ldots c_{j+t},\ldots )$ we establish that $c_{i+t}-c_i=c_{j+t}-c_j, i\neq j$, i.e.\ that it is not an orthomorphism, then all permutations with this subarray in its $q-1$-cycle are also non-orthomorphisms, which significantly reduces the number of computations.
%\bigskip

\section{Counting number of permutations in \texttt{C++}}

% TODO: vratio sam prvi dio , jer je visio dio teksta
As stated earlier, we want to find all the permutations $g$ of the field $\mathbb{F}_q$ with the property $g(0) = 0$ and $g$ is of the form $g = (c_0 \ c_1 \ \ldots \ c_{q-2})$ with $c_i \neq 0$ for $i \in \lbrace 0, \ldots, q-2 \rbrace$, (i.e.\ $g$ is a cycle of length $q-1$ which fixes the element $0$) such that $g(x) - x$ is a permutation.
To produce all permutations $g=(c_0 \ c_1 \ \ldots \ c_{q-2})$, $g(0)=0$, with the desired property, it is sufficient to generate all cycles $(c_0 \ c_1 \ldots \ c_{q-2})$, and to verify if $g(x) - x$ is injective, or equivalently, to check if $c_{(k+j) \mod (q-1)} - c_{k}$  are all distinct, for each $k \in \lbrace 0, \ldots, q-2 \rbrace$, and for each iteration $j \in \lbrace 1, \ldots, q-2 \rbrace$.
Note that we allow here for $\mathbb{F}_q$ to be any finite field, i.e.\ $q = p^n$ for some prime $p$ and natural $n$, and use natural identification between elements of $\mathbb{F}_q$ and elements of set $\lbrace 0, \ldots, q-1 \rbrace$ and identification $\mathbb{F}_{p^n}=\mathbb{F}_q^n$.
\par
Obviously, we can (and will) assume $c_0 = 1$. 
\par
The naive algorithm for finding all such permutations would be to generate all permutations $g$ of the cycle $(1 \ 2 \ \ldots \ q-1)$, and to check whether the permutation $g^j - \mathrm{id}$ is injective for all $j \in \lbrace 1, \ldots, q-2 \rbrace$.
In that way, we obtain the complexity $(q-2)! (q-2)(q-1)$. Caution should be exercised for last two terms, since we do not know how many 
possible $(q-2)(q-1)$ checks would be performed for each permutation.
\par
The first optimization is to compare subarrays as mentioned above.
\par
The second optimization is more inclined towards $\texttt{C++}$. We use bitwise 'or' and bitwise shift operators for checking the injectivity part, and to keep information on whether the permutation so far is injective (for fixed $j$), one machine word (32 bits on most modern computers) is enough to keep the information about the (partial) set of image values -- instead of $q$ bytes or $q$ words.
Since a rollback operation -- undoing data of injectivity -- can happen and happens in each level of $1$ to $q-2$, every improvement that can be done in a partial step (verifying injectivity; rolling back) of scale $c$ factors to up $c^{q-2}$.
\par
In terms of generalization, we note that our algorithm can find the permutations for \emph{any} field $\mathbb{F}_q$, since the only specific part regarding the underlying field is the (pregenerated) table of subtraction. 
This has to be done only once for a target field $\mathbb{F}_q$.
As a side note, our benchmarks showed that certain (but small) speed improvements have already been achieved introducing the subtraction operation table, so it is already a welcome improvement.
A number of orthomorphisms for fields of order less than $30$ are given in the table, and this matches the lower bound given in Theorem~\ref{thm:22}.

\begin{tabular}{|c|c|c|c|c|c|c|c|c|c|c|c|c|c|c|c|}
\hline
q & 3 &4&5&7&8&9&11&13&16&17&19&23&25&27&29\\
\hline 
number of $g(x)$&1&2&2&2&48&12 &4&4&2688 &8&6&10&80 &1728&12\\ 	
\hline 
\end{tabular}

Thus, for $q\leq 30$, all Costas polynomials are of the form $L(x^s)$ where $L$ is a linearized permutation polynomial and $\gcd(q-1,s)=1$.

\section{Some additional properties of Costas polynomials}

\begin{Prop} Let $f(x)$ be a Costas polynomial. Let $d\in \mathbb{F}_q\setminus \{1\}$ and let $g_d(x)=f(df^{-1}(x))$.  We let $h_d(x)=f(dx)-f(x)$. Then the following holds:
\begin{itemize} 

\item[a)] The function $h_d(x)+h_d(dx)+\cdots +h_d(d^sx)=h_{d^{s+1}}(x)$ is a permutation polynomial for all $s$ such that $d^{s+1}\neq 1$. If $d^{s+1}=1$, then $h_d(x)+h_d(dx)+\cdots +h_d(d^sx)=0$ .
\item[b)] $g_t\circ g_s=g_{ts}$.
\end{itemize}
\end{Prop}
\begin{proof}
\begin{itemize}

\item[a)] $h_d(x)+h_d(dx)+\cdots +h_d(d^sx)=(f(dx)-f(x))+(f(d^2x)-f(dx))+\cdots +(f(d^{s+1}x)-f(d^sx))=f(d^{s+1}x)-f(x)=h_{d^{s+1}}(x)$, $d^{s+1} \neq 1$. If $d^{s+1}=1$, then we have 
\begin{multline*}
    h_d(x)+h_d(dx)+\cdots +h_d(d^sx) = \\
    (f(dx)-f(x))+(f(d^2x)-f(dx))+\cdots +(f(d^{s+1}x)-f(d^sx)). 
\end{multline*}
This telescopes to $=f(d^{s+1}x)-f(x)=f(x)-f(x)=0.$
\item[b)] $g_t\circ g_s(x)=f(tf^{-1}(f(sf^{-1}(x))))=f(tsf^{-1}(x))=g_{ts}(x).$

\end{itemize}
\end{proof} 

\begin{Prop} Let $f(x)=\sum_{i=0}^{q-2}x^i$ be a Costas polynomial over a field of odd characteristic. Then $a_i\neq 0$ implies $a_{q-1-i}=0$ for all $i=1,2,\ldots ,q-2$.
\end{Prop}
\begin{proof}Let $G(x)=\sum_{i=0}^{q-1}b_ix^i$. Then the coefficient of $x^{q-1}$ in $G(x)^2$ equals $\sum_{i=0}^{q-1}b_ib_{q-1-i}$. If $G(x)$ is a permutation polynomial, then by the Hermite criterion we have $\sum_{i=0}^{q-1}b_ib_{q-1-i}=0$ (see \citep{Lidl}). We apply this to Costas polynomials. If $f(x)$ is Costas, then $f(d x)-f(x)=\sum_{i=0}^{q-1}a_i(d^i-1)x^i$ is a permutation polynomial for all $d\in \mathbb{F}_q\setminus \{1\}$. 
% TODO: Da li [] ima neko posebno znacenje?
Let $G(x)=[f(d x)-f(x)]$. The coefficient with $x^{q-1}$ in $(G(x))^2=[f(d x)-f(x)]^2$ is 
$$\sum_{i=0}^{q-1}a_ia_{q-1-i}(d^i-1)(d^{q-1-i}-1).$$
Expanding yields $$\sum_{i=0}^{q-1}a_ia_{q-1-i}(d^{q-1}+1)-2\sum_{i=0}^{q-1}a_ia_{q-1-i}d^i=0,$$
$ d \neq 1$.
Viewed as a polynomial in $d $, this expression vanishes for all $d \neq 1$. 
Therefore, it must be either the zero polynomial or a scalar multiple of $(d^q-d)\cdot(d-1)^{-1}=d^{q-1}+d^{q-2}+\cdots +d$. 

Since $f(x)$ is a permutation polynomial, we also have $\sum_{i=0}^{q-1}a_ia_{q-1-i}=0$, so the first term vanishes and we obtain
$$-2\sum_{i=0}^{q-1}a_ia_{q-1-i}d^i=0.$$ This implies $a_ia_{q-1-i}=0$ for all $i$, and hence  $a_i\neq 0$ implies $a_{q-1-i}=0$. 
\end{proof}
Therefore, we conclude that Costas polynomials have at least half of their coefficients equal to zero.

In all cases of orthomorphisms with fixed point $0$ and a single cycle of length $q-1$, such that all of their compositions are also orthomorphisms, we have $g^{\frac{q-1}{2}}(x)=-x$ in fields of order $q\leq 30$. If $f=L(x^s)$ and $g(x)=L(\alpha^s L^{-1}(x))$, then $g^{\frac{q-1}{2}}(x)=L(-L^{-1}(x))=-L(L^{-1}(x))=-x$. This motivates the following proposition.
\begin{Prop} For a Costas polynomial $f(x)$, the mapping $g(x)=f(\alpha f^{-1}(x))$ has the property $g^{\frac{q-1}{2}}(x)=-x$ if and only if $f(x)$ is an odd function. 
\end{Prop} 
\begin{proof} If $g^{\frac{q-1}{2}}(x)=-x$ then $f(-f^{-1}(x))=-x$  which implies $-f^{-1}(x)=f^{-1}(-x)$. Let $x=f(y)$. Then $-f^{-1}(f(y))=f^{-1}(-f(y))$, i.e.\ $f^{-1}(-f(y))=-y$ and hence  $-f(y)=f(-y)$. Thus, $f(x)$ is an odd function.
 
Conversely, if $f(x)$ is odd, then $g(x)=f(\alpha f^{-1}(x))$ is also odd (as a composition of odd functions), and all of its compositions are odd functions. In particular, $g^{\frac{q-1}{2}}(-x)=f(-f^{-1}(x))=-f(f^{-1}(x))=-x$. 
\end{proof}

\section{Construction of Latin Squares}

A \emph{planar polynomial} $f(x) \in \mathbb{F}_q[x]$ is a polynomial such that $f(x+a)-f(x)$ is a permutation polynomial for every $a \in \mathbb{F}_q^*$. Planar polynomials cannot be bijective.  It is known that if $f(x)$ is planar, then the family of Latin squares defined by
$\mathcal{L}^a_{ij}=i+f(j+a)-f(j)$ forms a complete family of pairwise mutually orthogonal Latin squares (see \citep{Coulter}, Theorem 5.3).
Costas polynomials are the multiplicative analogue of planar polynomials.

\begin{Them}
Let $f(x) \in \mathbb{F}_q[x]$ be a Costas polynomial. Then the family of Latin squares defined by
\[
\mathcal{L}^d_{i,j}=i+f(dj)-f(j), \quad d \neq 1
\]
is a complete
family of pairwise mutually orthogonal Latin squares.
\end{Them}

\begin{proof} Note first that there are $q-1$ of such Latin squares. We compare ordered pairs of elements  in the Latin squares $\mathcal{L}^d $ and $\mathcal{L}^s$ where $d\neq s, d,s\neq 1$. Assume
$$(i+f(dj)-f(j), i+f(sj)-f(j))=(k+f(dt)-f(t), k+f(st)-f(t)), i,j,k,t\in \mathbb{F}_q.$$
Then subtracting the equalities $$i+f(dj)-f(j)=k+f(dt)-f(t)$$ and $$i+f(sj)-f(j)=k+f(st)-f(t)$$ we obtain
$$f(dj)-f(sj)=f(dt)-f(st).$$
Using the substitutions $z=sj, y=st$ we obtain $f(ds^{-1}z)-f(z)=f(ds^{-1}y)-f(y)$. 
The polynomial $f(ds^{-1}x)-f(x)$ is a permutation polynomial, and thus $z=y$ which implies $j=t$ and, consequently, $i=k$. Therefore, $\mathcal{L}^d$ and $\mathcal{L}^s$ are mutually orthogonal.
\end{proof}

The function $g_\alpha (x)=L(\alpha L^{-1}(x))$, which is obtained from known Costas polynomials, is a strong complete mapping and can be used to construct Knut Vik designs and complete families of MOLS (see \citep{Amela}).

 Furthermore, on the multiplicative group $\mathbb{F}_q^*$, the orthomorphism $f(df^{-1}(x))=\mathcal{L}_{x,d}$, where $f(x)$ is a Costas polynomial, defines a Latin square ($d=y$) where rows and columns are denoted by $x,d\in \mathbb{F}_q^*$. This Latin square is equivalent to the known Latin square $K_{x,y}=xy$. On the other hand, if the rows are labeled by $d\in \mathbb{F}_q\setminus \{1\}$ and the columns by $x\in \mathbb{F}_q^*$, then we can also obtain the  Latin square $\mathcal{L}_{d,x}=f(df^{-1}(x))-x$ of order $q-1$ with elements in $\mathbb{F}_q^*$. 

\section*{Code availability}

The implementation used in this paper is publicly available at
\url{https://github.com/samarski/costas}. 
A permanent archived version corresponding to this manuscript is available through Zenodo at DOI: \url{https://doi.org/10.5281/zenodo.20549103}.

\end{document}